\documentclass[12pt,a4paper]{article}%
\usepackage[utf8]{inputenc}
\usepackage{hyperref}
\usepackage{amsmath}
\usepackage{amsfonts}
\usepackage{amssymb}
\usepackage{xcolor}
\usepackage[space]{grffile}
\usepackage{graphicx}%
\providecommand{\U}[1]{\protect\rule{.1in}{.1in}}
\newtheorem{theorem}{Theorem}

\newtheorem{corollary}[theorem]{Corollary}

\newtheorem{lemma}[theorem]{Lemma}

\newtheorem{proposition}[theorem]{Proposition}

\newtheorem{observation}[theorem]{Observation}

\newenvironment{proof}[1][Proof]{\noindent\textbf{#1.} }{\ \hfill \rule{0.5em}{0.5em}\bigskip}
\graphicspath{{D:/Dropbox/Riste-Sedlar/Yu Yang/Paper2/Slike/}}

\begin{document}

\title{The subpath number of cactus graphs}
\author{Martin Knor$^{1}$, Jelena Sedlar$^{2,4}$, Riste \v{S}krekovski$^{3,4,5}$, Yu
Yang$^{6}$\\{\small $^{1}$ \textit{Slovak University of Technology, Bratislava, Slovakia
}}\\[0.1cm] {\small $^{2}$ \textit{University of Split, FGAG, Split, Croatia }}\\[0.1cm] {\small $^{3}$ \textit{University of Ljubljana, FMF, Ljubljana,
Slovenia }}\\[0.1cm] {\small $^{4}$ \textit{Faculty of Information Studies, Novo Mesto,
Slovenia }}\\[0.1cm] {\small $^{5}$ \textit{Rudolfovo - Science and Technology Centre, Novo
mesto}\textit{, Slovenia }}\\[0.1cm] {\small $^{6}$ \textit{School of Software, Pingdingshan University,
Pingdingshan, China}}}
\maketitle

\begin{abstract}
The subpath number of a graph $G$ is defined as the total number of subpaths
in $G$, and it is closely related to the number of subtrees, a well-studied
topic in graph theory. This paper is a continuation of our previous paper
\cite{Knor2025}, where we investigated the subpath number and identified
extremal graphs within the classes of trees, unicyclic graphs, bipartite
graphs, and cycle chains. Here, we focus on the subpath number of cactus
graphs and characterize all maximal and minimal cacti with $n$ vertices and
$k$ cycles. We prove that maximal cacti are cycle chains in which all interior
cycles are triangles, while the two end-cycles differ in length by at most
one. In contrast, minimal cacti consist of $k$ triangles, all sharing a common
vertex, with the remaining vertices forming a tree attached to this joint
vertex. By comparing extremal cacti with respect to the subpath number to
those that are extremal for the subtree number and the Wiener index, we
demonstrate that the subpath number does not correlate with either of these
quantities, as their corresponding extremal graphs differ.

\end{abstract}

\section{Introduction}

The number of non-empty subtrees, denoted by $N(G)$, in a graph $G$ was first
studied in the context of trees \cite{Szekely2005}. Various properties of
$N(G)$ have been explored for several subclasses of trees \cite{Kirk2008,
Szekely2007, Zhang2013}, and more recently, research on the number of subtrees
has been extended to certain classes of general graphs \cite{Cacti2022,
Xu2021}. For further interesting results on this topic, we refer the reader to
\cite{Czabarka2008, YuYang1, YuYang2, YuYang3}. The subtree number can attain
extremely high values, making its exact evaluation challenging. Motivated by
this, in our previous paper \cite{Knor2025}, we introduced the so-called
subpath number, which seems to be more tractable.

The subpath number of a graph $G$ is defined as the total number of all
subpaths in $G$, including those of zero length. For this graph invariant, we
established several preliminary results, such as the exact value of the
subpath number for trees and unicyclic graphs, as well as its behavior under
edge insertion, which led to identifying the extremal graphs among all
connected graphs on $n$ vertices. Additionally, we examined bipartite graphs
and cycle chains, determining the extremal graphs within each of these
families. Moreover, for cycle chains, we provided an exact formula for the
subpath number.

The Wiener index $W(G)$ of a graph $G$ is defined as the sum of the distances
over all pairs of vertices in $G$. Introduced in Wiener's seminal paper
\cite{Wiener}, it was originally shown to correlate with the chemical
properties of certain molecular compounds. Since then, the Wiener index has
become one of the most extensively studied indices in chemical graph theory;
for an overview of key results, we refer the reader to the surveys
\cite{risteSurvey1, risteSurvey2}. Interestingly, a "negative" correlation has
been observed between the number of subtrees and the Wiener index: in many
graph classes, the graph that maximizes the number of subtrees also minimizes
the Wiener index, and vice versa. This phenomenon has been noted in several
graph families, including cactus graphs \cite{Xu2022}.

Motivated by these observations, in this paper, we investigate the behavior of
the subpath number in the class of cactus graphs, where both the number of
vertices and the number of cycles are prescribed. We fully characterize the
cacti that minimize and maximize the subpath number. Furthermore, we
demonstrate that the subpath number does not exhibit a direct correlation with
either the Wiener index or the subtree number, as the extremal cactus graphs
with respect to the subpath number differ from those in the other two cases.
This suggests that the subpath number is an interesting graph invariant worthy
of independent study.

\section{Preliminaries}

For a graph $G$, the \emph{subpath number} is defined as the number of paths
in $G$, including trivial paths of length $0$. The subpath number of a graph
$G$ is denoted by $\mathrm{pn}(G)$. Before delving into specific cases, let us
first explore some fundamental properties of this quantity.

We begin by considering trees with $n$ vertices. It is well known that every
pair of vertices in a tree is connected by a unique path, which leads to the
following observation.

\begin{observation}
\label{Obs_tree} If $T$ is a tree on $n$ vertices, then ${\mathrm{pn}}(T) =
\binom{n+1}{2}$.
\end{observation}

Since unicyclic graphs are obtained from trees by introducing a single edge, a
natural next step is to analyze unicyclic graphs with $n$ vertices. In a
unicyclic graph $G$ containing a cycle of length $g$, removing the edges of
the cycle results in precisely $g$ connected components. The following result,
established in \cite{Knor2025}, provides an explicit formula for the subpath
number of such graphs.

\begin{proposition}
\label{Prop_unicyclic} Let $G$ be a unicyclic graph on $n$ vertices with the
cycle $C$ of length $g$. Denote by $n_{1},n_{2},\dots,n_{g}$ the number of
vertices in the components that remain after removing all edges of $C$. Then,
\[
{\mathrm{pn}}(G)=n+2\binom{n}{2}-\binom{n_{1}}{2}-\binom{n_{2}}{2}%
-\dots-\binom{n_{g}}{2}.
\]

\end{proposition}

This result immediately allows us to determine the subpath number for cycles
on $n$ vertices, as well as the extremal unicyclic graphs.

\begin{corollary}
\label{Cor_cycle} For a cycle on $n$ vertices, we have ${\mathrm{pn}}(C_{n}) =
n^{2}$.
\end{corollary}

\begin{corollary}
\label{Cor_unicyclic} Among unicyclic graphs with $n$ vertices, ${\mathrm{pn}%
}(G)$ attains its maximum value if and only if $G = C_{n}$, and its minimum
value if and only if the only cycle in $G$ is a triangle with two of its
vertices having degree two.
\end{corollary}

Next, we examine how the subpath number behaves when an edge is removed. The
following lemma, established in \cite{Knor2025}, formalizes this observation.

\begin{lemma}
\label{Lemma_aditivity} Let $G$ be a connected graph on $n$ vertices, and let
$e$ be an edge of $G$. Let $G^{\prime}$ be the graph obtained from $G$ by
removing the edge $e$. Then,
\[
{\mathrm{pn}}(G^{\prime}) < {\mathrm{pn}}(G).
\]

\end{lemma}

This result immediately implies the characterization of extremal graphs with
respect to the subpath number among all connected graphs with $n$ vertices, as
stated in the following theorem from \cite{Knor2025}.

\begin{theorem}
Let $G$ be a connected graph on $n$ vertices. Then
\[
\binom{n}{2} \leq{\mathrm{pn}}(G) \leq\frac{n!}{2} \sum_{i=0}^{n-1} \frac
{1}{i!} + \frac{n}{2},
\]
where the lower bound is attained if and only if $G$ is a tree, and the upper
bound if and only if $G = K_{n}$.
\end{theorem}

\section{Maximal cacti with respect to the subpath number}

Denote by $\mathcal{C}_{n,k}$ the class of all cactus graphs on $n$ vertices
with $k$ cycles. In this section we will characterize cactus graphs from
$\mathcal{C}_{n,k}$ with the maximum value of the subpath number. In the next
section we will do the same for the minimum value of the subpath number. Once
we do that, we can compare the extremal cacti for subpath number with the
extremal cacti for the Wiener index and the number of subtrees. It turns out
that the subpath number is not correlated with either of these quantities.

Let us first define two notions we use in this subsection. A \emph{bridge} in
a graph $G$ is any edge $e$ of $G$ such that $G-e$ has more connected
components than $G.$ A graph is \emph{bridgeless} if it has no bridges. Notice
that a bridgeless graph does not contain leaves. In order to characterize
cactus graphs from $\mathcal{C}_{n,k}$ with the maximum value of the subpath
number, we will use three cactus transformations. In the first transformation
we will address bridges of a cactus graph, in the second the incidence
structure of the cycles, and in the third the size of the cycles.

\begin{lemma}
\label{Lemma_max1}If a cactus graph $G\in\mathcal{C}_{n,k}$ contains a bridge,
then there exists a bridgeless cactus graph $G^{\prime}\in\mathcal{C}_{n,k}$
such that $\mathrm{pn}(G^{\prime})>\mathrm{pn}(G)$.
\end{lemma}

\begin{proof}
Let $e=uv$ be a bridge of $G$ such that one of its end-vertices belongs to a
cycle $C,$ say $v\in V(C).$ If $G$ has bridges, such an edge $e$ must exist.
Let $w$ be the neighbor of $v$ on $C,$ and let $G^{\prime}$ be the graph
obtained from $G$ by removing the edge $vw$ and adding the edge $uw$ instead.
Obviously, $G^{\prime}\in\mathcal{C}_{n,k}$ and $G^{\prime}$ has one bridge
less than $G.$

Let us prove that $\mathrm{pn}(G^{\prime})>\mathrm{pn}(G)$. To see this,
denote by $\mathcal{P}$ (resp. $\mathcal{P}^{\prime}$) the set of all paths in
$G$ (resp. $G^{\prime}$). We partition the set $\mathcal{P}$ into three parts,
$\mathcal{P}_{1}$ contains all paths $P$ of $\mathcal{P}$ such that $uv\in
E(P)$ and $vw\in E(P),$ $\mathcal{P}_{2}$ all paths $P$ of $\mathcal{P}$ with
$uv\not \in E(P)$ and $vw\in E(P),$ and $\mathcal{P}_{3}$ all the remaining
paths of $\mathcal{P}$. Next, a function $f:\mathcal{P}\rightarrow
\mathcal{P}^{\prime}$ is defined as follows:

\begin{itemize}
\item If $P\in\mathcal{P}_{1}$ then $f(P)\in\mathcal{P}^{\prime}$ is a path
obtained from $P$ by replacing subpath $uvw$ by the edge $uw.$ Notice that
$f(P)$ contains $uw,$ but not $uv.$

\item If $P\in\mathcal{P}_{2}$ then $f(P)\in\mathcal{P}^{\prime}$ is a path
obtained from $P$ by replacing the edge $vw$ by the subpath $vuw.$ Notice that
$f(P)$ contains both $uw$ and $uv.$

\item If $P\in\mathcal{P}_{3}$, then $f(P)=P.$ Notice that $f(P)$ does not
contain $uw.$
\end{itemize}

Let us show that $f$ is an injection. For that purpose, let $P_{1}$ and
$P_{2}$ be two distinct paths in $\mathcal{\mathcal{P}}$. If $P_{1}%
\in\mathcal{P}_{i}$ and $P_{2}\in\mathcal{P}_{j}$ for $i<j$, then $f(P_{1})$
and $f(P_{2})$ differ either in the edge $uw$ or in the edge $uv$. On the
other hand, if $P_{1}$ and $P_{2}$ belong to the same $\mathcal{P}_{i}$, then
they differ in the part of the path not changed by the function $f,$ so
$f(P_{1})$ and $f(P_{2})$ are distinct in $G^{\prime}$. Hence, in all cases we
have established $f(P_{1})\not =f(P_{2}),$ so $f$ is an injection.

Let us next show that $f$ is not a surjection. To see this, notice that $uv$
is the only path connecting $u$ and $v$ in $G,$ and $f(u,v)=uv.$ On the other
hand, there are two paths connecting $u$ and $v$ in $G^{\prime}.$ Since the
function $f$ preserves the end-vertices of a path, this implies that $f$ is
not a surjection.

We have established a function $f$ between $\mathcal{P}$ and $\mathcal{P}%
^{\prime}$ which is injection, but not surjection, so we can conclude that
$\left\vert \mathcal{P}\right\vert <\left\vert \mathcal{P}^{\prime}\right\vert
$ which means $\mathrm{pn}(G)<\mathrm{pn}(G^{\prime})$. Applying repeatedly
this transformation until we obtain a graph without bridges yields the claim
of the lemma.
\end{proof}

The above lemma confirms that, in the case of unicyclic graphs, the subpath
number attains the maximum value only for the cycle $C_{n},$ which is already
established in Corollary \ref{Cor_unicyclic}. In what follows, we will deal
with the cacti with at least two cycles.

A vertex $v$ of a bridgeless cactus graph $G$ is an \emph{intersection vertex}
if $v$ belongs to at least two cycles of $G$. Let $\mathcal{V}$ be the set of
all intersection vertices of $G$ and let $\mathcal{C}$ be the set of all
cycles of $G$. A \emph{cycle-incidence graph} $T_{G}$ of a bridgeless cactus
graph $G$ is defined as the graph on the set of vertices $\mathcal{V\cup
\mathcal{C}}$ such that an intersection vertex $v\in\mathcal{V}$ and a cycle
$C\in\mathcal{C}$ are connected by an edge in $T_{G}$ if $v$ belongs to the
cycle $C$ in $G.$ Notice that $T_{G}$ is a tree and every leaf in $T_{G}$ is a
vertex of $\mathcal{C}$. A \emph{cactus chain} is a bridgeless cactus graph
$G$ such that $T_{G}$ is a path.

\begin{lemma}
\label{Lemma_max2}Let $G\in\mathcal{C}_{n,k}$ be a bridgeless cactus graph
with $k\geq2$. If $G$ is not a cactus chain, then there exists a cactus chain
$G^{\prime}\in\mathcal{C}_{n,k}$ such that $\mathrm{pn}(G^{\prime
})>\mathrm{pn}(G).$
\end{lemma}

\begin{proof}
Let $t$ be a vertex of $T_{G}$ of degree at least three. Since $T_{G}$ is not
a path, such a vertex $t$ must exist. A component of $T_{G}-t$ which does not
contain a vertex of degree $\geq3$ is called a \emph{thread}. Denote by
$T_{1},T_{2},\ldots,T_{k}$ all the components of $T_{G}-t.$ We may assume that
$t$ is chosen so that $T_{1},\ldots,T_{k-1}$ are all threads. Since $k\geq3,$
this implies $T_{1}$ and $T_{2}$ are both threads. Let $V(T_{i})$ denote the
set of vertices of $G$ contained in cycles whose corresponding vertices of
$T_{G}$ belong to $T_{i}.$ Since $T_{1}$ and $T_{2}$ are both threads, we may
assume $\left\vert V(T_{1})\right\vert \leq\left\vert V(T_{2})\right\vert .$

We now introduce a graph transformation of $G$ into $G^{\prime}$ which, as we
will establish, increases the subpath number. If $t\in\mathcal{V}$, then set
$u=t$ and denote by $C$ the cycle of $T_{k}$ containing $u$. On the other hand
if $t\in\mathcal{C}$, then denote by $u$ a vertex of $T_{k}$ incident with $t$
and denote by $C$ the cycle of $T_{k}$ containing $u$.
Let $v$ and $w$ be the neighbors of $u$ on $C$ in $G.$ Further, let $z$ be a
vertex of degree $2$ on the cycle in $G$ which corresponds to the leaf of
$T_{1}$ in $T_{G}.$ Graph $G^{\prime}$ is defined as the graph obtained from
$G$ by removing edges $vu$ and $wu,$ and adding edges $vz$ and $wz$ instead.
Notice that $G^{\prime}$ belongs to $\mathcal{C}_{n,k}.$

We show that $\mathrm{pn}(G)<\mathrm{pn}(G^{\prime}).$ For a pair of vertices
$x,y\in V(G)$, denote by $\mathrm{pn}_{G}(x,y)$ the number of paths in $G$
which connect vertices $x$ and $y.$ Denote further $\Delta(x,y)=\mathrm{pn}%
_{G^{\prime}}(x,y)-\mathrm{pn}_{G}(x,y),$ and notice that
\[
\mathrm{pn}(G^{\prime})-\mathrm{pn}(G)=\sum_{x,y\in V(G)}\Delta(x,y).
\]
Observe that $\Delta(x,y)$ may be negative only if $x$ belongs to a cycle of
$T_{k}$ and $y$ belongs to a cycle of $T_{1}$. Since $\left\vert
V(T_{1})\right\vert \leq\left\vert V(T_{2})\right\vert $, there exists an
injection $f:V(T_{1})\rightarrow V(T_{2})$. It is sufficient to prove that
$\Delta(x,y)+\Delta(x,f(y))>0,$ for any $x\in V(T_{k})$ and $y\in V(T_{1})$.

For a pair of vertices $x,y\in V(G)$, we define the corresponding path
$P_{x,y}$ in $T_{G}$. Recall that every cycle of $G$ has a corresponding
vertex in $T_{G}$. Hence, we will refer by "cycle" to the cycle of $G$ as well
as to the corresponding vertex of $T_{G}$. Assume that $x$ belongs to a cycle
$C_{x}$ of $G$ and $y$ belongs to $C_{y}$. If neither $x$ nor $y$ are
intersection vertices in $G$, then $P_{x,y}$ is the path of $T_{G}$ which
connects the cycles $C_{x}$ and $C_{y}$. If $x$ (resp. $y$) is an intersection
vertex of $G$, then $P_{x,y}$ starts with $x$ (resp. ends with $y$) in $T_{G}%
$. In other words, if $x$ (resp. $y$) belongs to more than one cycle of $G,$
then $P_{x,y}$ in $T_{G}$ may contain only one cycle of $G$ to which $x$
(resp. $y$) belongs. Next, for a pair of vertices $x,y\in V(G),$ the number
$c(x,y)$ is defined as the number of cycles on $P_{x,y}$.

Let us proceed with proving that $\Delta(x,y)+\Delta(x,f(y))>0,$ for any $x\in
V(T_{k})$ and $y\in V(T_{1}).$ Denote by $c$ the number of cycles in $T_{1}.$
If $t$ is an intersection vertex, we have
\begin{align*}
\Delta(x,y)+\Delta(x,f(y))  &  \geq2^{c(x,t)}(2-2^{c}+2^{c(t,f(y))+c}%
-2^{c(t,f(y))})\\
&  =2^{c(x,t)}\big((2^{c}-1)(2^{c(t,f(y))}-1)+1\big)>0.
\end{align*}
If $t$ is a cycle, we have
\begin{align*}
\Delta(x,y)+\Delta(x,f(y))  &  \geq2^{c(x,t)}(2-2^{c+1}+2^{c+1+c(t,f(y))}%
-2^{1+c(t,f(y))})\\
&  =2^{c(x,t)}(2^{c}-1)(2^{1+c(t,f(y))}-2)>0,
\end{align*}
since $c\geq1$ and $c(t,f(y))\geq1$. Hence, we have established that
$\mathrm{pn}(G)<\mathrm{pn}(G^{\prime})$. Notice that the sum of the degrees
over all vertices of $T_{G}$ of degree at least $3$ has decreased from $G$ to
$G^{\prime},$ so applying this transformation repeatedly yields a bridgeless
cactus chain $G^{\prime}$.
\end{proof}

\begin{figure}[h]
\begin{center}
\raisebox{-0.9\height}{\includegraphics[scale=0.6]{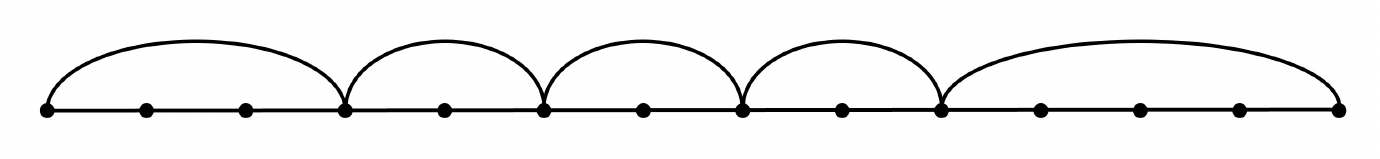}}
\end{center}
\caption{A pseudo triangle chain $\mathrm{PTC}(14,5).$}%
\label{Fig_PTC}%
\end{figure}

A cycle $C$ of a bridgeless cactus graph $G$ is an \emph{end-cycle} if at most
one of its vertices has degree greater than two, otherwise $C$ is an
\emph{interior} cycle. Notice that a cactus chain $G\in\mathcal{C}_{n,k}$ with
$k\geq2$ contains precisely two end-cycles. A \emph{pseudo triangle chain,}
denoted by $\mathrm{PTC}(n,k),$ is a cactus chain from $\mathcal{C}_{n,k}$ in
which every interior cycle is a triangle and the two end-cycles differ in the
number of vertices by at most one. This notion is illustrated by Figure
\ref{Fig_PTC}.

\begin{lemma}
\label{Lemma_max3}Let $G\in\mathcal{C}_{n,k}$ be a bridgeless cactus chain,
where $k\geq2$. If $G$ is distinct from $\mathrm{PTC}(n,k),$ then
$\mathrm{pn}(\mathrm{PTC}_{n,k})>\mathrm{pn}(G)$.
\end{lemma}

\begin{proof}
Assume first that there exists an interior cycle $C$ of $G$ which is not a
triangle. Since $G$ is a cactus chain, $C$ contains precisely two intersection
vertices. Hence, there exists a vertex $u$ on $C$ which is not an intersection
vertex. Denote by $v$ and $w$ the two neighbors of $u$ in $G,$ and notice that
both $v$ and $w$ belong to the cycle $C.$ Denote by $G_{1}$ and $G_{2}$ the
two connected components of $G-V(C),$ where we may assume the components are
denoted so that $\left\vert V(G_{1})\right\vert \leq\left\vert V(G_{2}%
)\right\vert .$ Notice that each of the components $G_{1}$ and $G_{2}$ must
contain vertices of precisely one end-cycle of $G.$ Let $a$ be a vertex of the
end-cycle of $G$ which is contained in $G_{1}$ such that $a$ is not an
intersection vertex of $G.$ Denote by $b$ a neighbor of $a.$ Let $G^{\prime}$
be a graph obtained from $G$ by removing edges $uv$, $uw$ and $ab,$ and adding
edges $ua$, $ub$ and $vw$ instead. Notice that $G^{\prime}$ is also a
bridgeless cactus chain on $n$ vertices with $k$ cycles.

We wish to establish that $\mathrm{pn}(G)<\mathrm{pn}(G^{\prime})$. Again,
denote $\Delta(x,y)=\mathrm{pn}_{G^{\prime}}(x,y)-\mathrm{pn}_{G}(x,y)$ and
notice that $\Delta(x,y)$ may be negative only if $x=u$ and $y\in V(G_{1}).$
Since $\left\vert V(G_{1})\right\vert \leq\left\vert V(G_{2})\right\vert ,$
there exists an injection $f:V(G_{1})\rightarrow V(G_{2}).$ Denote by $c$ the
number of cycles in $G_{1}.$ Notice that for each $y\in V(G_{1})$ it holds
that%
\begin{align*}
\Delta(x,y)+\Delta(x,f(y))  &  \geq2-2^{c+1}+2^{c(x,f(y))+c}-2^{c(x,f(y))}\\
&  =(2^{c}-1)(2^{c(x,f(y))}-2)>0,
\end{align*}
since $c\geq1$ and $c(x,f(y))\geq2$. Applying repeatedly this transformation
yields a bridgeless cactus chain $G^{\prime}$ in which all interior cycles are
triangles, such that $\mathrm{pn}(G)<\mathrm{pn}(G^{\prime}).$ Assume that
cycles of $G^{\prime}$ are denoted by $C_{1},\ldots,C_{k}$ such that $C_{1}$
and $C_{k}$ are end-cycles, and the pair of cycles $C_{i}$ and $C_{i+1}$ share
an intersection vertex for every $i=1,\ldots,k-1.$ If the end-cycles $C_{1}$
and $C_{k}$ of $G^{\prime}$ differ in the number of vertices by at most one,
we are done. So, let us assume that $\left\vert V(C_{k})\right\vert
-\left\vert V(C_{1})\right\vert \geq2.$

Let $u$ be a vertex of $C_{k}$ which is not an intersection vertex, and let
$v$ and $w$ be the two neighbors of $u$ in $G^{\prime}.$ Let $a$ be a vertex
of $C_{1}$ which is not an intersection vertex and let $b$ be a neighbor of
$a.$ Denote by $G^{\prime\prime}$ the graph obtained from $G^{\prime}$ by
removing edges $uv$, $uw$ and $ab$ and adding edges $vw,$ $ua$ and $ub$ instead.

We prove that $\mathrm{pn}(G^{\prime})<\mathrm{pn}(G^{\prime\prime})$. Denote
$\Delta(x,y)=\mathrm{pn}_{G^{\prime\prime}}(x,y)-\mathrm{pn}_{G^{\prime}%
}(x,y)$ and notice that $\Delta(x,y)$ may be negative only if $x=u$ and
$y\in\big(V(C_{1})\cup\cdots\cup V(C_{k-1})\big)\setminus V(C_{k})$. Since
$\left\vert V(C_{1})\right\vert <\left\vert V(C_{k})\right\vert $ and every
interior cycle of $G^{\prime}$ is a triangle, there exists an injection%
\[
f:V(C_{1})\cup\cdots\cup V(C_{k-1})\rightarrow V(C_{2})\cup\cdots\cup
V(C_{k})
\]
such that $f(y)\in V(C_{k+1-i})$ if and only if $y\in V(C_{i})$. Observe that
if $y$ is the intersection vertex of $V(C_{i})\cap V(C_{i+1})$ then $f(y)$ is
the intersection vertex of $C_{k+1-i}$ with $C_{k-i}$. Assume that $y\in
V(C_{i})\setminus\big(V(C_{i+1}\cup V(C_{i-1})\big)$ for $1\leq i\leq k-1$.
Then, $\mathrm{pn}_{G^{\prime}}(x,y)=2^{k-(i-1)}$ and $\mathrm{pn}%
_{G^{\prime\prime}}(x,y)=2^{i}$. Also, it holds that $\mathrm{pn}_{G^{\prime}%
}(x,f(y))=2^{k-(k+1-i-1)}$ and $\mathrm{pn}_{G^{\prime\prime}}(x,y)=2^{k+1-i}%
.$ Hence, we have%
\[
\Delta(x,y)+\Delta(x,f(y))=2^{i}-2^{k-(i-1)}+2^{k+1-i}-2^{k-(k+1-i-1)}=0.
\]
Now assume that $y\in V(C_{i})\cap V(C_{i-1})$ for $2\leq i\leq k-1$. Then
$\mathrm{pn}_{G^{\prime}}(x,f(y))=2^{k-(k+1-i)}$ and $\mathrm{pn}%
_{G^{\prime\prime}}(x,f(y))=2^{k+1-i}$. Hence, we have
\[
\Delta(x,y)+\Delta(x,f(y))=2^{i-1}-2^{k-(i-1)}+2^{k+1-i}-2^{k-(k+1-i)}=0.
\]
Also, $\left\vert V(C_{1})\right\vert <\left\vert V(C_{k})\right\vert $
implies that $f$ is not a surjection, so there exists a vertex $z\in V(C_{k})$
which is not in the image of $f.$ For such a vertex $z,$ we have%
\[
\Delta(u,z)=2^{k}-2>0,
\]
since $k>2$. We conclude that $\mathrm{pn}(G^{\prime})<\mathrm{pn}%
(G^{\prime\prime})$. Applying this transformation repeatedly yields the graph
$\mathrm{PTC}(n,k)$, and we are done.
\end{proof}

Lemmas \ref{Lemma_max1}-\ref{Lemma_max3} immediately yield the following result.

\begin{theorem}
\label{Tm_max} The graph $\mathrm{PTC}(n,k)$ uniquely maximizes the subpath
number among all cacti on $n$ vertices with $k\geq2$ cycles.
\end{theorem}

By Theorem~{\ref{Tm_max}}, it is useful to calculate the subpath number of
$\mathrm{PTC}(n,k)$.

\begin{lemma}
\label{Lemma_PTC} The subpath number of $\mathrm{PTC}(n,k)$ equals
\[
(n^{2}-4kn+14n+4k^{2}-28k+49)2^{k-2}+\tfrac{1}{2}(n^{2}-4kn-6n+4k^{2}%
-4k+7)+\delta,
\]
where $\delta=0$ if $n$ is odd, and $\delta=1-2^{k-2}$ if $n$ is even.
\end{lemma}

\begin{proof}
Denote by $C_{1},C_{2},\dots,C_{k}$ the cycles in the cactus chain
$\mathrm{PTC}(n,k)$, so that $C_{i}$ and $C_{i+1}$ have a vertex in common,
say $w_{i}$, where $1\le i\le k-1$. Assume that $|V(C_{1})|\ge|V(C_{k})|$.
Then $|V(C_{1})|=n_{1}=\lceil\frac{n-2k+5}2\rceil$, $|V(C_{k})|=n_{2}%
=\lfloor\frac{n-2k+5}2\rfloor$, and $|V(C_{i})|=3$ if $2\le i\le k-1$.

We count the number of $u-v$ paths when $u$ and $v$ are in a common cycle,
then the number of $u-v$ paths when $u$ and $v$ are in neighboring cycles but
not in their intersection, then the number of $u-v$ paths when $u\in V(C_{i})$
and $v\in V(C_{i+2})$ but $u,v\notin V(C_{i+1})$, etc. However, we sum the
number of paths in the opposite order.

Observe that if $u\in V(C_{i})\setminus\{w_{i}\}$ and $v\in V(C_{j}%
)\setminus\{w_{j-1}\}$, where $1\leq i<j\leq k$, then $\mathrm{PTC}(n,k)$
contains $2^{j-i+1}$ paths connecting $u$ with $v$ since in every cycle
$C_{i},C_{i+1},\dots,C_{j}$ we can choose one of the two possibilities of how
to traverse it. So we have
\begin{align*}
\mathrm{pn}(\mathrm{PTC}(n,k) ={}  &  (n_{1}-1)(n_{2}-1)2^{k}+\sum_{i=1}%
^{k-2}2(n_{1}-1)2^{i+1}+\sum_{i=1}^{k-2}2(n_{2}-1)2^{i+1}\\
&  +\sum_{i=1}^{k-3}2\cdot2\cdot2^{i+1}(k-2-i)+n_{1}^{2}+n_{2}^{2}%
+(k-2)3^{2}-(k-1),
\end{align*}
where the last term $-(k-1)$ appears since in $n_{1}^{2}+n_{2}^{2}+(k-2)3^{2}$
we counted the paths of length $0$ consisting of cut-vertices $w_{1}%
,w_{2},\dots,w_{k-1}$ twice.

Since $\sum_{i=1}^{t}2^{i}=2^{t+1}-2$, the sum of the first two sums is
$(n_{1}+n_{2}-2)(2^{k+1}-8)$. And since $\sum_{i=1}^{t}i\cdot2^{i-1}%
=t\cdot2^{t+1}-(t+1)2^{t}$, the third sum equals $(k-2)(2^{k+1}%
-16)-(k-3)2^{k+2}+(k-2)2^{k+1}=2^{k+2}-16k+32$. So the expression for the
subpath number of $\mathrm{PTC}(n,k)$ reduces to
\begin{align*}
\mathrm{pn}(\mathrm{PTC}(n,k)={}  &  (n_{1}-1)(n_{2}-1)2^{k}+(n_{1}%
+n_{2}-2)(2^{k+1}-8)+2^{k+2}-16k+32\\
&  +n_{1}^{2}+n_{2}^{2}+8k-17.
\end{align*}

Now if $n$ is odd, we get
\begin{align*}
\mathrm{pn}(\mathrm{PTC}(n,k) ={}  &  \bigg(\frac{n-2k+3}{2}\bigg)^{2}%
\cdot2^{k}+(n-2k+3)(2^{k+1}-8)+2^{k+2}\\
&  +2\bigg(\frac{n-2k+5}{2}\bigg)^{2}-8k+15,
\end{align*}
which reduces to
\begin{align*}
\mathrm{pn}(\mathrm{PTC}(n,k)={}  &  (n^{2}-4kn+14n+4k^{2}-28k+49)2^{k-2}\\
&  +\tfrac{1}{2}(n^{2}-4kn-6n+4k^{2}-4k+7).
\end{align*}

On the other side when $n$ is even, we get
\begin{align*}
\mathrm{pn}(\mathrm{PTC}(n,k)={}  &  \bigg(\frac{n-2k+4}{2}\bigg)\bigg(\frac
{n-2k+2}{2}\bigg)\cdot2^{k}+(n-2k+3)(2^{k+1}-8)+2^{k+2}\\
&  +\bigg(\frac{n-2k+6}{2}\bigg)^{2}+\bigg(\frac{n-2k+4}{2}\bigg)^{2}-8k+15,
\end{align*}
which reduces to
\begin{align*}
\mathrm{pn}(\mathrm{PTC}(n,k)={}  &  (n^{2}-4kn+14n+4k^{2}-28k+48)2^{k-2}\\
&  +\tfrac{1}{2}(n^{2}-4kn-6n+4k^{2}-4k+9).
\end{align*}

\end{proof}

\section{Minimal cacti with respect to the subpath number}

After we have characterized maximal cacti from $\mathcal{C}_{n,k}$ with
respect to the subpath number, our next goal is to establish minimal cacti in
the same class. We will use the same approach of graph transformation, in a
way that we will first address the size of cycles of $G$ by the transformation
inverse to the one of Lemma \ref{Lemma_max1}, which creates additional
bridges. Then we will address the interior cycles of $G$ and thus arrive to
all the minimal cacti in $\mathcal{C}_{n,k}$.

\begin{lemma}
\label{Lemma_min1}Let $G\in\mathcal{C}_{n,k}$ be a cactus graph. If $G$
contains a cycle which is not a triangle, then there exists a cactus graph
$G^{\prime}\in\mathcal{C}_{n,k}$ in which every cycle is a triangle such that
$\mathrm{pn}(G)>\mathrm{pn}(G^{\prime}).$
\end{lemma}

\begin{proof}
Let $C$ be a cycle of $G$ which is not a triangle and let $uv$ be an edge of
$C.$ Denote by $w$ the other neighbor of $u$ on $C.$ Let $G^{\prime}$ be the
graph obtained from $G$ by removing the edge $uv$ and adding the edge $vw.$ We
may consider that the graph $G$ is obtained from $G^{\prime}$ by removing the
edge $vw$ from it and adding the edge $uw$ instead. Then Lemma
\ref{Lemma_max1} implies $\mathrm{pn}(G)>\mathrm{pn}(G^{\prime}).$ Applying
the transformation repeatedly yields the result.
\end{proof}

An end-cycle of $G$ which is a triangle will be called an \emph{end-triangle}.
Observe that end-triangle has only one vertex whose degree is greater than
$2$. We have reduced the problem of finding minimal cacti to the class of
cacti in which every cycle is a triangle. Let us further show that no triangle
of minimal cacti can be interior, i.e., minimal cacti have only end-triangles.

\begin{lemma}
\label{Lemma_min2}Let $G\in\mathcal{C}_{n,k}$ be a cactus graph in which every
cycle is a triangle. If there exists an interior triangle in $G,$ then there
exists a cactus graph $G^{\prime}\in\mathcal{C}_{n,k}$ in which every cycle is
an end-triangle such that $\mathrm{pn}(G)>\mathrm{pn}(G^{\prime}).$
\end{lemma}

\begin{proof}
Let $C=u_{1}u_{2}u_{3}u_{1}$ be an interior triangle of $G.$ We may assume
that the degrees of the vertices $u_{1}$ and $u_{2}$ on the cycle $C$ are
greater than $2.$ Denote by $G_{i}$ the connected component of $G-E(C)$ which
contains the vertex $u_{i},$ for $i=1,2,3.$ Let $G^{\prime}$ be a graph
obtained from $G$ by removing the edge $xu_{2}$ and adding the edge $xu_{1},$
for every vertex $x$ of $G_{2}$ adjacent to $u_{2}$. Notice that
$\mathrm{pn}_{G}(a,b)$ increases only if $a=u_{2}$ and $b\in V(G_{2}%
)\setminus\{u_{2}\}$. But if $b\in V(G_{2})\setminus\{u_{2}\}$ then
$\mathrm{pn}_{G}(b,u_{2})+\mathrm{pn}_{G}(b,u_{1})=\mathrm{pn}_{G^{\prime}%
}(b,u_{1})+\mathrm{pn}_{G^{\prime}}(b,u_{2})$. And since for $a\in
V(G_{2})\setminus\{u_{2}\}$ and $b\in V(G_{1})\setminus\{u_{1}\}$ we have
$\mathrm{pn}_{G}(a,b)>\mathrm{pn}_{G^{\prime}}(a,b)$, we conclude that
$\mathrm{pn}(G)>\mathrm{pn}(G^{\prime})$.
Applying the transformation repeatedly yields the result.
\end{proof}

\begin{figure}[h]
\begin{center}%
\begin{tabular}
[t]{lll}%
\raisebox{-0.9\height}{\includegraphics[scale=0.6]{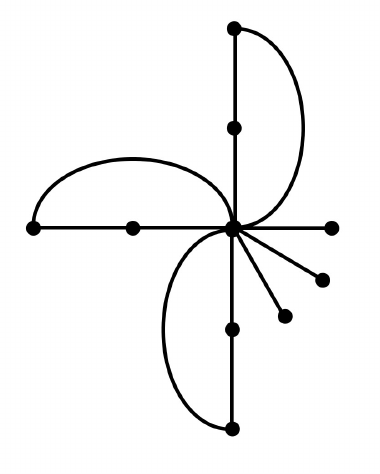}} &
\raisebox{-0.9\height}{\includegraphics[scale=0.6]{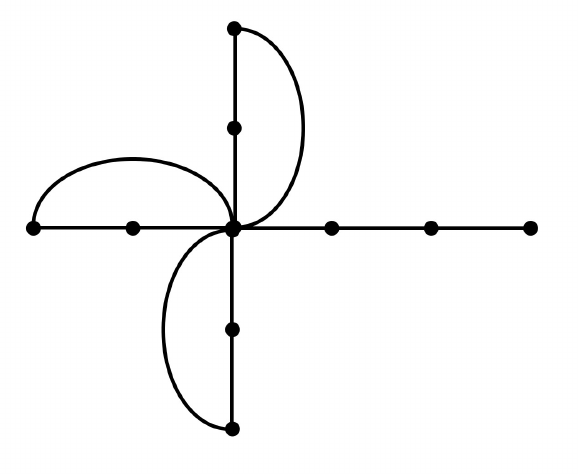}} &
\raisebox{-0.9\height}{\includegraphics[scale=0.6]{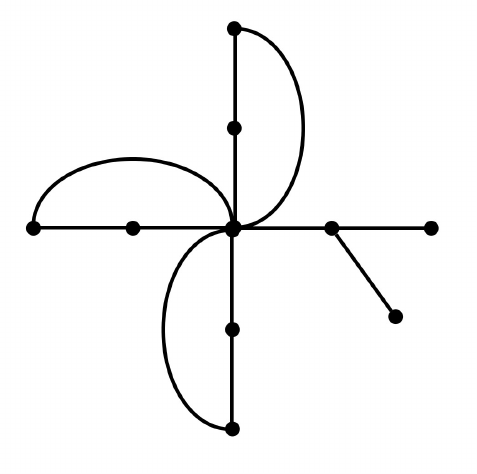}}
\end{tabular}
\end{center}
\caption{The figure shows three distinct cactus graphs from $\mathcal{C}%
_{10,3}.$ All these graphs minimize the subpath number in $\mathcal{C}%
_{10,3}.$ The leftmost graph is the pseudo friendship graph $\mathrm{PFG}%
(10,3)$ and only this graph minimizes Wiener index and maximizes the number of
subtrees over $\mathcal{C}_{10,3}.$}%
\label{Fig_PFG}%
\end{figure}

Figure \ref{Fig_PFG} shows three distinct graphs from $\mathcal{C}_{10,3}$ in
which every cycle is an end-triangle. In the next theorem we show that all
such graphs have the same subpath number, so they all minimize the subpath number.

\begin{theorem}
\label{Tm_min} A cactus graph $G\in\mathcal{C}_{n,k}$ has a minimum possible
value of the subpath number if and only if every cycle of $G$ is an end-triangle.
\end{theorem}

\begin{proof}
Lemmas \ref{Lemma_min1} and \ref{Lemma_min2} imply that it is sufficient to
establish that all cactus graphs of $\mathcal{C}_{n,k}$ in which every cycle
is an end-triangle have the same value of the subpath number. To see this, let
$G\in\mathcal{C}_{n,k}$ be such a cactus graph, and let us partition the set
of vertices $V_{1}$ and $V_{2},$ so that $V_{1}$ consists of all vertices of
cycles in $G$ which have the degree two and $V_{2}=V(G)\backslash V_{1}$.
Notice that $\mathrm{pn}_{G}(x,y)=1$ if and only if both $x$ and $y$ belong to
$V_{2}$. Further, $\mathrm{pn}_{G}(x,y)=2$ for $x\in V_{1}$ and either $y\in
V_{2}$ or $y$ belongs to the same triangle as $x$. Finally, $\mathrm{pn}%
_{G}(x,y)=4$ if $x,y\in V_{1}$ such that $x$ and $y$ belong to distinct
triangles. We conclude that
\[
\mathrm{pn}(G)=n+\binom{\left\vert V_{2}\right\vert }{2}+2k+2\left\vert
V_{1}\right\vert \left\vert V_{2}\right\vert +4\binom{k}{2}\cdot4.
\]
Plugging in $\left\vert V_{1}\right\vert =2k$ and $\left\vert V_{2}\right\vert
=n-2k,$ we obtain%
\[
\mathrm{pn}(G)=2k^{2}+2kn-5k+\frac{1}{2}n^{2}+\frac{1}{2}n.
\]
Since the expression for $\mathrm{pn}(G)$ depends only on the number of
vertices and cycles, we are done.
\end{proof}

\bigskip

Now that we have characterized cactus graphs from $\mathcal{C}_{n,k}$ which
maximize and minimize the subpath number, we can summarize our results in the
following corollary, which is a direct consequence of Theorems \ref{Tm_max}
and \ref{Tm_min}.

\begin{corollary}
\label{Cor_cactiPDS}For a cactus graph $G\in\mathcal{C}_{n,k},$ it holds that
\begin{align*}
\tfrac{1}{2}(n^{2}+4kn+n+4k^{2}-10k)\leq\mathrm{pn}(G)\leq\  &  (n^{2}%
-4kn+14n+4k^{2}-28k+49)2^{k-2}\\
&  +\tfrac{1}{2}(n^{2}-4kn-6n+4k^{2}-4k+7)+\delta,
\end{align*}
where $\delta=0$ if $n$ is odd, and $\delta=1-2^{k-2}$ if $n$ is even. The
left inequality is attained if and only if every cycle of $G$ is an
end-triangle, and the right inequality is attained if and only if
$G=\mathrm{PTC}(n,k).$
\end{corollary}

\bigskip

\begin{figure}[h]
\begin{center}
\includegraphics[scale=0.6]{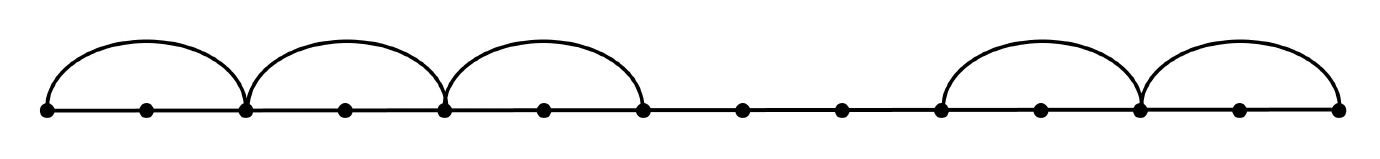}
\end{center}
\caption{The balanced saw graph $\mathrm{BSG}(14,5).$}%
\label{Fig_BSG}%
\end{figure}

Let us now compare extremal cacti with respect to the subpath number to those
with respect to the Wiener index and the number of subtrees. To do that, we
first need to introduce some particular cacti from the class $\mathcal{C}%
_{n,k}$. The \emph{balanced saw} graph $\mathrm{BSG}(n,k)$ is a cactus graph
from $\mathcal{C}_{n,k}$ obtained by joining a vertex of an end of a triangle
chain with $\left\lceil k/2\right\rceil $ cycles to a vertex of a triangle
chain with $\left\lfloor k/2\right\rfloor $ cycles by a path with $n-2k-2$
interior vertices. This notion is illustrated by Figure \ref{Fig_BSG}. A
\emph{pseudo friendship} graph $\mathrm{PFG}(n,k)$ is a cactus graph from
$\mathcal{C}_{n,k}$ obtained from $k$ triangles, all sharing a common vertex,
and $n-2k-1$ pendant edges attached to the same vertex. An example of a pseudo
friendship graph is the leftmost graph from Figure \ref{Fig_PFG}.

In \cite{CactiWienerMinimum,Gutman2017} the following result on the Wiener
index is established.

\begin{theorem}
\label{Tm_cactiWiener} The graph $\mathrm{BSG}(n,k)$ uniquely maximizes and
the graph $\mathrm{PFG}(n,k)$ uniquely minimizes the Wiener index among all
cacti from $\mathcal{C}_{n,k}$.
\end{theorem}

In \cite{Cacti2022, CactiWienerMinimum} the following result is established
regarding the subtree index, which is yet another fact supporting the
observation that minimal graphs for the Wiener index maximize the subtree
index and vice versa.

\begin{theorem}
\label{Tm_cactiSubtree} The graph $\mathrm{PFG}(n,k)$ uniquely maximizes and
the graph $\mathrm{BSG}(n,k)$ uniquely minimizes the subtree index among all
cacti from $\mathcal{C}_{n,k}$.
\end{theorem}

Comparing extremal graphs from Theorems \ref{Tm_cactiWiener} and
\ref{Tm_cactiSubtree} with the extremal graphs from Corollary
\ref{Cor_cactiPDS}, it is observable that the graph $\mathrm{BSG}(n,k)$ which
uniquely maximizes the Wiener index is distinct from the graph $\mathrm{PTC}%
(n,k)$ which uniquely maximizes the subpath number. As for the graph
$\mathrm{PFG}(n,k)$ which uniquely minimizes the Wiener index (resp. uniquely
maximizes the subtree index), it minimizes the subpath number also, but not
uniquely, as there are many more graphs which minimize the subpath number and
which are not minimal with respect to the Wiener index, see Figure
\ref{Fig_PFG}.

\bigskip

\vskip1pc \noindent\textbf{Acknowledgments.}~~  This work is partially
supported by Slovak research grants VEGA 1/0069/23, VEGA 1/0011/25,
APVV-22-0005 and APVV-23-0076, by Slovenian Research and Innovation Agency
ARIS program\ P1-0383, project J1-3002 and the annual work program of
Rudolfovo, by Project KK.01.1.1.02.0027 co-financed by the Croatian Government
and the European Union through the European Regional Development Fund - the
Competitiveness and Cohesion Operational Programme, by bilateral
Croatian-Slovenian project BI-HR/25-27-004 and by the Key Scientific and
Technological Project of Henan Province, China (grant nos. 242102521023), the
Youth Project of Humanities and Social Sciences Research of the Ministry of
Education (No. 24YJCZH199).



\end{document}